\documentclass[12pt,oneside,reqno,fleqn]{amsart}
\usepackage{amssymb}
\usepackage{bbm}
\usepackage{cases}
\usepackage{amsmath,mathtools}
\usepackage{graphicx}
\usepackage{mathrsfs}
\usepackage{stmaryrd}
\usepackage{color}
\usepackage{soul}
\usepackage[dvipsnames]{xcolor}
\usepackage{amsfonts}
\usepackage{amsmath,amssymb,amsthm}
\usepackage{libertine}
\usepackage[T1]{fontenc}
\usepackage[libertine,varbb]{newtxmath}
\usepackage[colorlinks,linkcolor=NavyBlue,citecolor=Gray,urlcolor=Periwinkle]{hyperref}
\usepackage[shortlabels] {enumitem}
\usepackage[nameinlink,capitalize]{cleveref}
\usepackage{orcidlink}
\usepackage{microtype}
\usepackage[textsize=tiny
,disable
]{todonotes}
\numberwithin{equation}{section}
\newtheorem{theorem}{Theorem}[section]
\newtheorem{lemma}[theorem]{Lemma}
\newtheorem{proposition}[theorem]{Proposition}
\newtheorem{corollary}[theorem]{Corollary}
\theoremstyle{remark}

\newtheorem{definition}[theorem]{Definition}

\crefname{eqn}{Equation}{Equations}
\crefname{assumption}{Assumption}{Assumptions}
\crefname{innercustomthm}{Condition}{Conditions}
\crefrangelabelformat{innercustomthm}{#3#1#4-#5#2#6}

\newcommand\cee{{\mathcal E}}
\newcommand\cff{{\mathcal F}}
\newcommand\cgg{{\mathcal G}}

\newcommand\cmm{{\mathcal M}}
\newcommand\cpp{{\mathcal P}}
\newcommand\mE{{\mathbb E}}

\newcommand\E{\mE}
\newcommand\PP{{\mathbb P}}

\newcommand{\R}{{\mathbb R}}

\newcommand{\Rd}{{\R^d}}
\newcommand{\tand}{\quad\text{and}\quad}

\newcommand{\1}{{\mathbf 1}}

\newcommand{\wei}[1]{\langle#1\rangle}
\newcommand{\norm}[1]{{\left\vert\kern-0.25ex\left\vert\kern-0.25ex\left\vert #1 
    \right\vert\kern-0.25ex\right\vert\kern-0.25ex\right\vert}}
\DeclareMathOperator*{\esssup}{ess\,sup}
\renewcommand{\le}{\leq}
\renewcommand{\ge}{\geq}
\newcommand{\bmo}{{\mathrm{BMO}}}
\newcommand{\vmo}{{\mathrm{VMO}}}
\newcommand{\var}{{\mathrm{var}}}
\newcommand{\khoa}[2][]
{\todo[color=SeaGreen,caption={}, #1]{#2}}

\allowdisplaybreaks

\hypersetup{pdftitle={Maximal inequalities and weighted BMO processes}, pdfauthor={Khoa L\^e}}

\begin{document}
	\title{Maximal inequalities and weighted BMO processes}
	
\author{Khoa L{\^e} \orcidlink{0000-0002-7654-7139}}
\address{School of Mathematics, University of Leeds, U.K.
}
\email{k.le@leeds.ac.uk}
	\begin{abstract}
		For a general adapted integrable right-continuous with left limits (RCLL) process $(X_t)_{t\ge0}$ taking values in a metric space $(\mathcal E,d)$, we show (among other things) that for every $m\in(1,\infty)$ and every $\tau>0$
		$$
			\kappa_1(m)\|\sup_{t\in[0,\tau]}\mathbb{E}(d(X_{t-},X_\tau)|\mathcal F_t)\|_m\le  \|\sup_{t\in[0,\tau]}d(X_0,X_t)\|_m\le \kappa_2(m) \|\sup_{t\in[0,\tau]}\mathbb{E}(d(X_{t-},X_\tau)|\mathcal F_t)\|_m
		$$ 
		with some universal constants $\kappa_1(m),\kappa_2(m)$ independent from $\tau$ such that $\kappa_1(m)=O(1)$ and $\kappa_2(m)=O(m)$ as $m\to\infty$. 
		This is a probabilistic version of Fefferman--Stein estimate for the sharp maximal functions.
		While the former inequality is derived easily from Doob's martingale inequality, the later inequality is a consequence of John--Nirenberg inequalities for weighted BMO processes, which are obtained in this note. We explain how John--Nirenberg inequalities can be utilized to obtain inequalities for martingales, both old and new alike in a unified way.
		\bigskip
		

	\end{abstract}
	
	\maketitle
\section{Introduction and main results} 
\label{sec:introduction}
	Doob's maximal martingale inequality (\cite{MR0058896}) and Burkholder--Davis--Gundy inequality (\cite{MR0400380}) have ultimately become indispensable and in the development of probability theory. They are however only available within  martingales and their derivatives. 
	In most situations, it is possible to introduce some auxiliary martingales such that these inequalities can be applied. 
	A recent example is the stochastic sewing lemma from \cite{le2018stochastic} which has led to a wide range of new applications from regularization by noise \cite{athreya2020well}, stochastic numerics \cite{dareiotis2021quantifying,le2021taming,butkovsky2021approximation} to rough stochastic differential equations \cite{FHL21}. However, when applying to certain problems, one encounters many technical issues and some unexplanatory conditions. 
	This raises a question whether martingale method is  the correct toolbox for these problems, which we do not attempt to answer herein.
	Nevertheless, as an initiative step forward, we derive in this note  maximal inequalities which are valid for general adapted integrable  stochastic processes having RCLL (right-continuous with left limits) sample paths. 
	To state the results, we first fix some notation. Let $(\Omega,\cgg,\PP)$ be a probability pace equipped with a filtration $\{\cff_t\}_{t\ge0}$ satisfying the usual conditions. 
	For each stopping time $S$, $\E_S$ denotes the conditional expectation with respect to $\cff_S$ and for each $G\in \cgg$, $\PP_S(G):=\E_S(\1_G)$. For an integrable random variable $\Xi$, $\E_\cdot \Xi$ always denotes a RCLL version of $t\mapsto \E_t \Xi$.
	For each $m\in(0,\infty)$, we denote $\|\cdot\|_m=(\E|\cdot|^m)^{1/m}$ and $\|\cdot\|_\infty=\esssup_\omega|\cdot|$.
	We define the constants
	\begin{align}\label{def.M}
	 	(c_m)^m=m\left(1+\frac1m\right)^{(m+1)^2}
	 	\tand M=1+\sum_{m=1}^\infty \frac{(c_mm)^m }{m!}e^{-3m}.
	\end{align}
	Note that  $\sup_{m\ge1}c_m<\infty$ and $M$ is finite by the ratio test.
	\begin{theorem}\label{cor.maximaldagger}
		Let $(\cee,d)$ be a metric space and $X:[0,\infty)\times \Omega\to(\cee,d)$ be a RCLL adapted integrable process. We define for each $t\le \tau$,
		\begin{align}\label{def.dagger}
			X^\sharp_{t,\tau}=\sup_{s\in[t,\tau]}\E_sd(X_{s-}, X_\tau)
			\tand X^\sharp_\tau=X^\sharp_{0,\tau}.
		\end{align} 
		Then for every $\tau>0$ and $m\in(0,\infty)$, 
		\begin{gather}
			\label{est.upperbound}
			\|\sup_{t\in[0,\tau]}d(X_0,X_t)\|_m \le(2c_mm)
			\|\sup_{t\in[0,\tau]}\E_tX^\sharp_{t,\tau}\|_m,
			\\
			\sup_{\varepsilon>0}\left\|\sup_{t\in[0,\tau]}\frac{d(X_0,X_t)}{\varepsilon+\sup_{s\in[0,t]}\E_sX^\sharp_{s,\tau}}\right\|_m\le 2c_mm,
			\label{est.X2}
			\\ \sup_{\varepsilon>0}\E\exp\left(\lambda\sup_{t\in[0,\tau]}\frac{d(X_0,X_t)}{\varepsilon+\sup_{s\in[0,t]}\E_sX^\sharp_{s,\tau}}\right)\le M \quad\text{for every } \lambda\le (2e^3)^{-1}.
			\label{est.X3}
		\end{gather}
		In addition, 
		\begin{align}
			\frac{m-1}{2m-1}\|X^\sharp_\tau\|_m\le\|\sup_{t\in[0,\tau]}d(X_0,X_t)\|_m\le \frac{2c_mm^2}{m-1} 
			\|X^\sharp_\tau\|_m \quad \text{for every }m\in(1,\infty),
			\label{est.Smaximal}
		\end{align}
		and
		\begin{align}\label{est.Smax01}
			\|\sup_{t\in[0,\tau]}d(X_0,X_t)\|_m\le \frac{2c_mm}{(1-m)^{1/m}} 
			\|\E_0X^\sharp_\tau\|_m  \quad\text{for every }m\in(0,1).
		\end{align}
	\end{theorem}
	That $\sup_{t\in[0,\tau]}\E_tX^\sharp_{t,\tau}$ is a well-defined random variable needs some explanation. Observe that $t\mapsto \E_t X^\sharp_t$ is a supermartingale. In addition, $t\mapsto \E [\E_tX^\sharp_{t,\tau}]=\E\sup_{s\in[t,\tau]}\E_s d(X_{s-},X_\tau)$ is right-continuous by the Lebesgue monotone convergence theorem. Hence, be \cite[Chapter II (2.9)]{MR1725357}, the process $t\mapsto \E_t X^\sharp_t$ has a RCLL version for which we always use in \cref{cor.maximaldagger} and hereafter.
	This shows that $\sup_{t\in[0,\tau]}\E_tX^\sharp_{t,\tau}$ is measurable.

	When $X$ is a real valued discrete martingale, \eqref{est.Smaximal} can be traced back at least to Garsia \cite[Theorem III.5.2]{MR0448538} and Stroock \cite{MR341601}. 
	For general stochastic processes, \eqref{est.Smaximal} seems to be new and has no counterpart in literatures, as far as the author's knowledge.

	We note that the estimate $\|\sup_{t\in[0,\tau]}d(X_0,X_t)\|_m\le c(m)\|\sup_{t\in[0,\tau]}\E_td(X_{t},X_\tau)\|_m $ fails for $m>2$ even for discrete martingale. A counter example  given by Os\c{e}kowski in \cite{MR3426632} is the martingale $g=(0,g_1,g_1,g_1,\ldots)$ with $\PP(g_1=-\varepsilon)=1-\PP(g_1=\varepsilon^{-1})=(1+\varepsilon^2)^{-1}$.

	We will derive \cref{cor.maximaldagger} from John--Nirenberg inequalities for stochastic processes of bounded weighted mean oscillation (weighted BMO). 
	\begin{definition}\label{def.bmophi}
		Let $\tau>0$ be a fixed number,  $(\phi_t)_{t\in[0,\tau]}$ be a positive RCLL adapted  process and $(V_t)_{t\in[0,\tau]}$ be a real valued RCLL adapted process. $V$  belongs to $\bmo_\phi$ if \khoa{mind to remove $\tau$?}
		\begin{align*}
			[V]_{\bmo_\phi}:=  \sup_{0\le S\le T\le \tau}\|\phi^{-1}_S\E_S|V_T-V_{S-}|\|_\infty<\infty,
		\end{align*}
		where the supremum is taken over all stopping times $S,T$. 
		For a process $V$ in $\bmo_\phi$, its modulus of oscillation $\rho^\phi(V)$ is defined by
		\begin{align*}
			\rho^\phi_{s,t}(V)=\sup_{s\le S\le T\le t}\|\E_S(\phi^*_{S,T})^{-1}|V_T-V_{S-}|\|_\infty, \quad 0\le s\le t\le \tau,
		\end{align*}
		where $\phi^*_{s,t}=\sup_{r\in[s,t]}\phi_r$ and  $S,T$ denote generic stopping times. 
		We also define 
		\begin{align*}
			\kappa^\phi_{s,t}(V)=\lim_{h\downarrow0}\sup_{s\le u\le v\le t,v-u\le h}\rho^\phi_{u,v}(V).
		\end{align*}
	\end{definition}
	We note that $\rho^\phi$ is well-defined and 
	\begin{align*}
		\rho^\phi_{s,t}\le\sup_{s\le S\le T\le t} \|\phi_S^{-1}\E_S|V_T-V_{S-}|\|_\infty\le[V]_{\bmo_\phi}.
	\end{align*}
	The relation between \cref{cor.maximaldagger} and weighted BMO processes is that any RCLL adapted integrable stochastic process is BMO with a suitable weight. Indeed, let $X$ be as in \cref{cor.maximaldagger} and fix $\tau>0$.
	For stopping times $S\le T\le \tau$, we have 
	\begin{align*}
	 	\E_Sd(X_T,X_\tau)=\E_S\1_{(T<\tau)}\lim_{\varepsilon\downarrow0}\E_{T+\varepsilon}d(X_{T+\varepsilon-}, X_\tau)
	 	\le \E_S\sup_{S\le t\le \tau}\E_td(X_{t-},X_\tau)=\E_S X^\sharp_{S,\tau}
	\end{align*}
	and hence 
	\begin{align*}
		\E_S|d(X_0,X_T)-d(X_0,X_{S-})|\le \E_S d(X_T,X_\tau)+\E_Sd(X_{S-},X_\tau)
		\le 2\E_S X^\sharp_{S,\tau}.
	\end{align*}
	This shows that for every $\varepsilon>0$, $d(X_0,X_\cdot)\in\bmo_{2\E_\cdot X^\sharp_{\cdot, \tau}+\varepsilon}$ with $[X]_{\bmo_{2\E_\cdot X^\sharp_{\cdot, \tau}+\varepsilon}}\le 1$. 

	John--Nirenberg inequalities for weighted BMO processes are described in the following result.
		\begin{theorem}\label{thm.bmophi}
		Let $(V_{t})_{t\in[0,\tau]}$ be a process in $\bmo_\phi$.
		Then for every $r\in[0,\tau]$ and every $m\in(0,\infty)$, 
		\begin{gather}\label{ineq.JNweighted}
			\E_r\sup_{r\le t\le \tau}|V_t-V_{r}|^m\le (c_mm\rho^\phi_{r,\tau}(V))^m\E_r|\phi^*_{r,\tau}|^m \quad\text{a.s. }
			\\\shortintertext{and}
			\label{est.jnweight2}
			\E_r\left(\sup_{r\le t\le \tau}\frac{|V_t-V_{r}|}{\phi^*_{r,t}}\right)^m\le(c_mm\rho^\phi_{r,\tau}(V))^m \quad \text{a.s.}
		\end{gather}
		Consequently, 
		\begin{align}\label{est.jnexpwei}
			\sup_{r\in[0,\tau]}\left\|\E_r \exp\left(\lambda\sup_{r\le t\le \tau}\frac{|V_t-V_r|}{\phi^*_{r,t}}\right)\right\|_\infty<\infty \text{ for every } \lambda\le(e^3\kappa^\phi_{r,\tau}(V))^{-1}.
		\end{align}
		Additionally, \eqref{ineq.JNweighted}-\eqref{est.jnexpwei} hold with $V_{r-}$ in place of $V_r$.
	\end{theorem}
	Although not stated,
	\cref{cor.maximaldagger,thm.bmophi} have natural discrete analogues. 
	Neither one of \eqref{ineq.JNweighted} and \eqref{est.jnweight2} implies the other. However, \eqref{est.jnweight2} has an advantage of being meaningful even when $\phi^*_{r,\tau}$ is not $L^m$-integrable. While \eqref{ineq.JNweighted} may be comparable to \cite[Theorem 1]{MR2136865}, \cite[Theorem III.5.2]{MR0448538} and \cite[Lemma A.1.2]{MR2190038}, the estimates \eqref{est.jnweight2} and \eqref{est.jnexpwei} are genuine and inspired from similar phenomenon in the theory of singular integrals \cite{MR1913610,MR3124931}. Normalized estimates such as \eqref{est.jnexpwei} play important roles in the theory of self-normalized processes \cite{MR2488094} and self-normalized  large deviations \cite{MR1428510,MR2016616}.
	The proofs of \cref{cor.maximaldagger,thm.bmophi} are presented in \cref{sec:weighted_bmo_processes}.

	To put \cref{thm.bmophi} into a perspective, we apply it to martingales for which many classical inequalities are readily available. We rediscover in a systematic and natural way the close relations between a martingale with its square functions, which is the heart of the works of Burkholder, Davis and Gundy \cite{MR0400380,MR365692,davis1970intergrability}. Although we are not able to recover all of the known results solely on this method, \cref{thm.bmophi} turns out to be very robust and we discover new inequalities for martingales. These results are presented in \cref{sec:inequalities_for_martingales}. \cref{cor.maximaldagger} could also be used here, however, we have learned that applying \cref{thm.bmophi} directly to martingales gives better results.

	A discussion on weighted BMO processes would be incomplete without mentioning processes with vanishing weighted mean oscillation. Results for these processes are of independent interests and are presented in \cref{sec:weighted_vmo_spaces}. 
	For further applications of BMO processes, we refer to the companion paper \cite{le2022stochastic} and the references therein.
	Applications of \eqref{est.Smaximal} will be discussed elsewhere for conciseness.

	\textit{Relation to previous works.}
	The later estimate in \eqref{est.Smaximal} is closely related to  Fefferman--Stein's estimate for the sharp maximal function in \cite{MR447953}. Indeed, let $f:\Rd\to\R$ be a locally integrable function and define the  maximal functions
	\begin{align*}
		f^\sharp(x)=\sup_{Q\ni x}\frac1{|Q|}\int_{Q}|f(y)-f_Q|dy, \quad \cmm f(x)=\sup_{Q\ni x}\frac1{|Q|}\int_{Q}|f(y)|dy,
	\end{align*}
	where  $Q$ denotes a cube in $\Rd$, $|Q|$ is the Lebesgue measure of $Q$ and $f_Q=\frac1{|Q|}\int_Qf(y)dy$. Then \cite[Theorem 5]{MR447953} shows that
	\begin{align*}
		\|\cmm f\|_{L^p(\Rd)}\le A_p\|f^\sharp\|_{L^p(\Rd)}, \quad p\in(1,\infty),
	\end{align*}
	which is comparable to the second estimate in \eqref{est.Smaximal}.
	For further connections with the theory of singular integrals, we refer to \cite{MR447953,MR341601}, appendix A.1 of \cite{MR2190038},  remark III.5.1 of \cite{MR0448538} and chapter 3 of \cite{MR3617205}.

	Weighted BMO sequences and discrete martingales were considered earlier by Garsia \cite{MR0448538}, Stroock \cite{MR341601}, Stroock and Varadhan \cite{MR2190038}, and weighted BMO processes were considered by Geiss \cite{MR2136865}. 
	\khoa{also Garsia \cite{MR0448538} thm III.5.2, see Remark III.5.1, thm II.1.2} 
	Inequality \eqref{ineq.JNweighted} is known (at least implicitly) to Geiss. However, John--Nirenberg inequalities are formulated as equivalence of norms in \cite{MR2136865} and require an additional condition on the weight process $\phi$.
	In contrast, we have tried our best to frame the results in \cref{thm.bmophi} under simplest assumptions. This is important and beneficial because \cref{cor.maximaldagger} and martingale inequalities (\cref{sec:inequalities_for_martingales}) are consequences of \cref{thm.bmophi} but not earlier results.
	The relation between the exponential constant and modulus of mean oscillation has been observed in the author's previous work \cite{le2022stochastic} in the absence of weight. The proof of \cref{thm.bmophi} is different from \cite{le2022stochastic} but still follows similar arguments which appeared in earlier works. However our presentation simplifies previous proofs. 
	In literatures on probability theory, 
	John--Nirenberg inequalities are often derived for BMO martingales and their validity for other processes is rarely discussed. Some connections between weighted BMO martingales and BDG inequalities can be found in \cite{MR0448538} through a different point of view from the current article.

\section{Proofs of main results} 
\label{sec:weighted_bmo_processes}
	We will need the following elementary lemma.
	\begin{lemma}\label{lem.xy}
		If $X$ and $Y$ are nonnegative random variables satisfying
		\begin{align*}
			\PP(Y> \alpha+\beta)\le \theta\PP(Y>\alpha)+\PP(X>\theta\beta)
		\end{align*}
		for every $\alpha>0$, $\beta>0$ and $\theta\in(0,1)$; then for for every $m\in(0,\infty)$, 
		\begin{align*}
			\|Y\|_m\le c_m m\|X\|_m
		\end{align*}
		where the constant $c_m$ is defined in \eqref{def.M}.
	\end{lemma}
	\begin{proof}
		We choose $\beta= h \alpha$ for some $h>0$ and integrate the inequality with respect to $m \alpha^{m-1}d \alpha$ over $(0,k/(1+h))$ to get that
		\begin{align*}
			(1+h)^{-m} \int_0^km \alpha^{m-1}\PP(Y>\alpha)d \alpha
			&\le \theta\int_0^{k}m \alpha^{m-1}\PP(Y>\alpha)d \alpha
			\\&\quad+\int_0^\infty m \alpha^{m-1}\PP(X>\theta h  \alpha)d \alpha.
		\end{align*}
		Sending $k\to\infty$ and using the layer cake representation $\E X^m=\int_0^\infty m \alpha^{m-1}\PP(X>\alpha)d \alpha$, we obtain that
		\begin{align*}
			\left[(1+h)^{-m}-\theta\right]\E Y^m\le (\theta h)^{-m}\E X^m.
		\end{align*}
		We now choose $h=\frac1m$ and $\theta=\left(\frac m{m+1}\right)^{m+1}$ to obtain the result.		
	\end{proof} 
	We note that the condition in \cref{lem.xy} is satisfied if $\PP(Y>\alpha+\beta,\ X\le \theta \beta)\le \theta\PP(Y>\alpha)$ which is closely related to the condition for Burkholder's good $\lambda$-inequality \cite[Lemma 7.1]{MR365692}. The formulation in \cref{lem.xy} is more convenient for our  purpose. 

	\begin{proof}[\bf Proof of \cref{thm.bmophi}]
		We fix $r\in[0,\tau)$ and assume without loss of generality that $\rho^\phi_{r,\tau}(V)=1$.
		We put $V^*=\sup_{t\in[r,\tau]}|V_t- V_{r}|$ and $\phi^*=\sup_{t\in[r,\tau]} \phi_r$.
		Let $\alpha, \beta$ be two positive numbers and define
		\begin{align*}
			S=\tau\wedge\inf\{t\in[r,\tau]:|V_t- V_{r}|> \alpha\}, 
			\quad 
			T=\tau\wedge\inf\{t\in[r,\tau]:|V_t- V_{r}|> \alpha+\beta\}, 
		\end{align*}
		with the standard convention that $\inf(\emptyset)=\infty$. Clearly $S$ and $T$ are stopping times and $r\le S\le T\le \tau$.
		On the event $\{V^*> \alpha+\beta\}$, we have $|V_T- V_{r}|\ge \alpha+\beta$, $|V_S-V_r|\ge \alpha$ and $|V_{S-}- V_{r}|\le \alpha$.
		The last inequality needs some justification. By right-continuity, there is $\varepsilon>0$ such that $|V_{s}-V_r|<\alpha$ for every $s\in[r,r+\varepsilon]$, which implies that $S>r$. Then one has $|V_{S-}-V_{r}|\le \alpha$ by definition of $S$. (The case when $V_{r}$ is replaced by $V_{r-}$, the inequality $|V_{S-}-V_{r-}|\le \alpha$ is trivial if $S=r$.)

		Using the triangle inequality $|V_T-V_r|\le|V_T-V_{S-}|+|V_{S-}-V_r|$, this implies that
		\begin{align*}
			\{V^*> \alpha+\beta\}\subset\{|V_T-V_{S-}|\ge \beta, V^*> \alpha\}.
		\end{align*}
		It follows that for every $G\in\cff_r$ and every $\theta\in(0,1)$,
		\begin{align*}
			\PP(V^*>\alpha+\beta,\ G)
			&\le \PP(|V_T-V_{S-}|\ge \beta,\ V^*>\alpha,\ G)
			\\&\le \PP(|V_T-V_{S-}|\ge \theta^{-1}\phi^*_{S,T},V^*>\alpha,\ G)
			+\PP( \phi^*> \theta \beta,V^*>\alpha,\ G).
		\end{align*}
		By conditioning, noting that $\{V^*>\alpha\}$ is $\cff_S$-measurable, and using definition of $\bmo_\phi$, we have
		\begin{align*}
			\PP(|V_T-V_{S-}|\ge \theta^{-1}\phi^*_{S,T},V^*>\alpha,\ G)\le \theta\PP(V^*>\alpha,\ G).
		\end{align*}
		Hence, we obtain from the above that
		\begin{align*}
			\PP(V^*> \alpha+\beta,\ G)
			&\le \theta\PP(V^*> \alpha,\ G)+\PP( \phi^*> \theta \beta,\ G).
		\end{align*}
		Applying \cref{lem.xy},
		\begin{align*}
			\|V^*\1_G\|_m\le c_mm\|\phi^*\1_G\|_m.
		\end{align*}
		Since $G$ is arbitrary in $\cff_r$, we obtain \eqref{ineq.JNweighted}.

		To show \eqref{est.jnweight2}, we follow a similar argument. This time, for $\alpha,\beta>0$, we define
		\begin{align*}
			&S=\tau\wedge\inf\{t\in[r,\tau]:(\phi^*_{r,t})^{-1} |V_t-V_{r}|> \alpha\}, 
			\\&T=\tau\wedge\inf\{t\in[r,\tau]:(\phi^*_{r,t})^{-1} |V_t-V_{r}|> \alpha+\beta\}.
		\end{align*}
		We have by triangle inequality and monotonicity,
		\begin{align*}
			\frac{|V_T-V_{r}|}{\phi^*_{r,T}}\le \frac{|V_T-V_{S-}|}{\phi^*_{S,T}}+\frac{|V_{S-}-V_{r}|}{\phi^*_{r,S-}}
		\end{align*}
		so that putting $(V/\phi)^*=\sup_{r\le t\le \tau}|V_t-V_r|/\phi^*_{r,t}$, we have
		\begin{align*}
			\left\{\left(\frac V \phi\right)^* > \alpha+\beta\right\}\subset\left\{\frac{|V_T-V_{S-}|}{\phi^*_{S,T}}\ge \beta,\ \left(\frac V \phi\right)^*> \alpha\right\}.
		\end{align*}
		From here, we obtain \eqref{est.jnweight2} through \cref{lem.xy} using similar arguments as previously.

		Inequalities \eqref{ineq.JNweighted} and \eqref{est.jnweight2} with  $V_{r-}$ in place of $V_r$ are obtain analogously.

		For $\lambda  \kappa^\phi_{r,\tau}(V)<e^{-3}$, there is an $h_0>0$ such that $\lambda \rho^\phi_{s,t}(V)\le e^{-3}$ whenever $t-s\le h_0$ and $s,t\in[r,\tau]$. For such $s,t$, we have by Taylor's expansion and \eqref{est.jnweight2} that 
		\begin{align*}
			\left\|\E_s \exp\left({\lambda\sup_{u\in[s,t]}\frac{|V_u-V_s|}{\phi^*_{s,u}}}\right)\right\|_\infty
			&\le\sum_{m=0}^\infty \frac{\lambda^m}{m!}\E_s\left(\sup_{u\in[s,t]}\frac{|V_u-V_s|}{\phi^*_{s,u}}\right)^m
			\\&\le1+\sum_{m=1}^\infty \frac{\lambda^m}{m!}(c_mm)^m(\rho^\phi_{s,t}(V))^m\le M,
		\end{align*}
		where $M$ is defined in \eqref{def.M}.
		Putting $Z_t=\sup_{u\in[r,t]}\frac{|V_u-V_r|}{\phi^*_{r,u}}$, since 
		\begin{align*}
		 	Z_t-Z_s\le\sup_{u\in[s,t]}\frac{|V_u-V_s|}{\phi^*_{s,u}},
		 \end{align*}
		  the previous estimate also implies that $\|\E_s e^{\lambda(Z_t-Z_s)}\|_\infty\le M$ whenever $t-s\le h_0$.
		Now partition $[0,\tau]$ by points $0=t_0<t_1<\ldots<t_n=\tau$ so that $\max_{1\le k\le n}(t_k-t_{k-1})\le h_0$. Then
		\begin{align*}
			\E_r e^{\lambda(Z_\tau-Z_r)}
			= \E_r e^{\lambda({Z_{t_{n-1}}-Z_r})} e^{\lambda({Z_{t_n}-Z_{t_{n-1}}})}
			\le  \E_r e^{\lambda({Z_{t_{n-1}}-Z_r})}\|\E_{t_n} e^{\lambda({Z_{t_n}-Z_{t_{n-1}}})}\|_\infty.
		\end{align*}
		Iterating the previous inequality yields
		\begin{align*}
			\Big\|\E_r e^{\lambda({Z_\tau-Z_r})}\Big\|_\infty\le \Big\|\E_r e^{\lambda{(Z_{t_j}-Z_r)}}\Big\|_\infty\prod_{k=j+1}^n \Big\|\E_{t_{k-1}} e^{\lambda({Z_{t_k}-Z_{t_{k-1}}})}\Big\|_\infty
		\end{align*}
		where $j$ is such that  $t_{j-1}\le r<t_j$. This implies that  $\|\E_r e^{\lambda({Z_\tau-Z_r})}\|_\infty\le M^n$
		which shows  \eqref{est.jnexpwei}.
		To show \eqref{est.jnexpwei} with $V_{r-}$, we use $\bar Z_t=\sup_{u\in[r,t]}\frac{|V_u-V_{r-}|}{\phi^*_{r,u}}$ instead of $Z_t$ and follow the same arguments.
	\end{proof}

	\begin{proof}[\bf Proof of \cref{cor.maximaldagger}]
		Let $\phi$ be a RCLL version of  the process $t\mapsto 2\E_tX^\sharp_{t,\tau}$. We have seen earlier in \cref{sec:introduction} that for every $\varepsilon>0$, $V_{\cdot}:=d(X_0,X_{\cdot})$ belongs to $\bmo_{\phi+\varepsilon}$ with $[V]_{\bmo_{\phi+\varepsilon}}\le 1$. 
		Applying \eqref{ineq.JNweighted} and sending $\varepsilon\downarrow0$, we have for every $m\in(0,\infty)$ that
		\begin{align*}
		 	\|\sup_{t\in[0,\tau]}|V_t|\|_m\le (2c_mm)\|\sup_{t\in[0,\tau]}\E_tX^\sharp_{t,\tau}\|_m,
		\end{align*}
		which shows \eqref{est.upperbound}.
		\eqref{est.X2} is a direct consequence of \eqref{est.jnweight2} and \eqref{est.X3} follows from \eqref{est.X2}.
		When $m>1$, combining with the Doob's maximal inequality 
		\[
			\|\sup_{t\le \tau}\E_tX^\sharp_{t,\tau}\|_m\le \|\sup_{t\le \tau}\E_tX^\sharp_\tau\|_m\le \frac m{m-1}\|X^\sharp_\tau\|_m,
		\] we obtain the later inequality in \eqref{est.Smaximal}.  We observe that  by triangle inequality,
		\begin{align*}
			X^\sharp_\tau\le\sup_{t\in[0,\tau]}(\E_td(X_0,X_\tau)+d(X_0,X_t)).
		\end{align*}
		Hence, by the Doob's maximal inequality, we have $\|X^\sharp_\tau\|_m\le (m/(m-1)+1)\|\sup_{t\in[0,\tau]}d(X_0,X_t)\|_m$, showing the first inequality in \eqref{est.Smaximal}. 

		To show \eqref{est.Smax01}, we note that $(\E_tX^{\sharp}_{t,\tau})$ is a non-negative supermartingale. Hence by \cite[Chapter II (1.15)]{MR1725357} $\PP(\sup_{t\le \tau}\E_tX^\sharp_{t,\tau}>\alpha)\le \frac1 \alpha\E_0X^\sharp_{\tau}$. Applying \cite[Lemma 2.22]{kuhn2022maximal} we have
		\begin{align*}
			\|\sup_{t\le \tau}\E_tX^\sharp_{t,\tau}\|_m\le (1-m)^{-1/m}\|\E_0X^{\sharp}_\tau\|_m
		\end{align*}
		which, together with \eqref{est.Smaximal}, shows \eqref{est.Smax01}.  
	\end{proof}


\section{Inequalities for martingales} 
\label{sec:inequalities_for_martingales}	
	In order to reveal the role of weighted BMO processes, we apply \cref{thm.bmophi} to martingales and positive sequences.
	We discover new ranges for  classical inequalities from the works of Burkholder, Davis and Gundy.
	We start with the following result, which is the dual of Doob's martingale inequality.
	\begin{proposition}
		Let $z_1,z_2,\ldots$ be positive random variables. Then for every $m\in(0,\infty)$, we have
		\begin{align}\label{doob0}
			\|\sum_{i=1}^\infty\E_{i}z_i\|_m\le 2c_mm\|\sup_{k}\E_k\sum_{i=k}^\infty z_i\|_m.
		\end{align}
		If $m\in(1,\infty)$,
		\begin{align}\label{doob1}
			\|\sum_{i=1}^\infty\E_{i}z_i\|_m\le 2c_mmm'\|\sum_{i=1}^\infty z_i\|_m.
		\end{align}
	\end{proposition}
	\begin{proof} 
		Define $V_k=\sum_{i=1}^k\E_{i}z_i$. From $\E_{k}(V_n-V_{k-1})=\E_{k}\sum_{i=k}^nz_i$,
		we see that $(V_k)_{k\le n}$ is BMO with the weight $(\phi_k=\E_k\sum_{i=k}^n z_i)_{k\le n}$. Applying \cref{thm.bmophi}, we have for every $m\in(0,\infty)$
		\begin{align*}
			\|\sum_{i=1}^n\E_{i}z_i\|_m\le 2c_mm\|\sup_{k\le n}\E_k\sum_{i=k}^n z_i\|_m.
		\end{align*}
		Sending $n\to\infty$ yields the first claim. The second claim follows from the first and Doob's maximal inequality.	
	\end{proof}
	In \cite{MR0400380}, Burkholder--Davis--Gundy  show \eqref{doob1}  with a different constant for all $m\in[1,\infty)$ (note that the case $m=1$ is trivial). The estimate \eqref{doob0} appears to be new. We take the chance to mention that it is shown in \cite[Theorem 20.1]{MR365692} that
	\begin{align*}
		\|\sum_{i=1}^\infty\E_{i}z_i\|_m\ge 2^{-1/m} \|\sum_{i=1}^\infty z_i\|_m, \quad m\in(0,1).
	\end{align*}

	Let $f=(f_{k})_{k\ge0}$ be a square integrable discrete martingale, $f_{-1}=0$. Define for each $k\le n$
	\begin{gather*}
		df_n=f_n-f_{n-1},\quad f^*_{k,n}=\sup_{k\le j\le n}|f_j-f_k|, \quad f^*_n=f^*_{0,n},
		\\ [f]_{k,n}=\sum_{j=k}^{n}|df_j|^2, \quad [f]_n=[f]_{0,n},
		\\\wei{f}_{k,n}= \sum_{j=k+1}^{n+1}\E_{j-1}|df_j|^2,\quad \wei{f}_n=\wei{f}_{0,n}.
	\end{gather*}
	From the martingale property, we have the following relations 
	\begin{gather}\label{id.isometry}
	 	\E_i|f_k-f_{i-1}|^2= \E_i[f]_{i,k}=\E_i([f]_k-[f]_{i-1}),
	 	\\\shortintertext{and}
	 	\label{id.condiso}
	 	\E_{i}(\wei{f}_k-\wei{f}_{i-1})=\E_i\wei{f}_{i,k}=\E_i|f_{k+1}-f_i|^2=\E_i[f]_{i+1,k+1}.
	\end{gather}
	\begin{theorem}\label{thm.dismart}
		For $m\in(0,\infty)$
		\begin{gather}
			\sup_{k\le n}\|f^*_{k,n}\|_{2m}\le 4c_{2m}m\Big\|\sup_{k\le n}\E_k[f]_{k,n}\Big\|_{m}^{1/2}\le 4c_{2m}m\Big\|\sup_{k\le n}(\E_k\wei{f}_{k,n-1}+|df_k|^2)\Big\|_{m}^{1/2},
			\label{bdg.f10}
			\\\|[f]_{n}\|_m\le 2c_mm\Big\|\sup_{k\le n}\E_k|f^*_{k,n}|^2\Big\|_m,
			\label{bdg.fbk0}
			\\\|\wei{f}_{n}\|_m\le 2c_mm\min\left\{\Big\|\sup_{k\le n}\E_k[f]_{k+1,n+1}\Big\|_m,\ \Big\|\sup_{k\le n}\E_k|f^*_{k,n+1}|^2\Big\|_m\right\}.
			\label{bdg.fweisup0}
		\end{gather}
		If $m\in(1,\infty)$, $m'=m/(m-1)$, we have
		\begin{gather}
			\sup_{k\le n}\|f^*_{k,n}\|_m\le 4c_{2m}mm'\|[f]_{n}\|_{m}^{1/2},
			\label{bdg.f11}
			\\\sup_{k\le n}\|f^*_{k,n}\|_{2m}\le 4c_{2m}m\left((m')^{1/2}\|\wei{f}_{n-1}\|_{m}^{1/2}+\|\sup_{k\le n}|df_k|\|_{2m}\right),\label{bdg.rosenthal}
			\\\|[f]_n\|_m\le 8c_mmm'\|f^*_{n}\|^2_{2m},
			\label{bdg.fbk1}
			\\\|\wei{f}_n\|_m\le 8c_mmm'\min\left\{\|[f]_{n+1}\|_m,\ \|f^*_{n+1}\|^2_{2m}\right\}.
			\label{bdg.fweisup1}
		\end{gather}
		In addition, there is a universal constant $\lambda>0$ such that for every $n\ge k\ge0$ and each pair $(V_{k},\Phi_{k,n})$ in 
		$$
			\left\{\left(f_k,\sup_{i\le k}(\E_i[f]_{i,n})^{\frac12}\right),\left([f]_{k},\sup_{i\le k}\E_i|f^*_{i,n}|^2\right), \left(\wei{f}_{k},\sup_{i\le k}\min(\E_i[f]_{i+1,n+1},\E_i|f^*_{i,n+1}|^2)\right) \right\},
		$$
		we have $\sup_{\varepsilon>0}\E e^{\lambda\sup_{k\le n}\frac{|V_{k}|}{\varepsilon+\Phi_{k,n}}}\le M$.
	\end{theorem}
	\begin{proof}
		Because $\E_i[f]_{i,k}\le\E_i[f]_{i,n}$ whenever $i\le k\le n$, we deduce from \eqref{id.isometry} that $(f_k)_{k\le n}$ is BMO with  weight $\left((\E_k[f]_{k,n})^{1/2}\right)_{k\le n}$.
		Identity \eqref{id.isometry} can also be read from right to left. In particular, because $\E_i|f_k-f_{i-1}|^2\le \E_i|f^*_{i-1,n}|^2$, it implies that $([f]_k)_{k\le n}$ is BMO with weight $(\E_k|f^*_{k-1,n}|^2)_{k\le n}$.
		Similarly, \eqref{id.condiso} implies that $(\wei{f}_k)_{k\le n}$ is BMO with either weight $(\E_k[f]_{k+1,n+1})_{k\le n}$ or weight $(\E_k|f^*_{k,n+1}|^2)_{k\le n}$.
		Applying \cref{thm.bmophi}, we obtain \eqref{bdg.f10}-\eqref{bdg.fweisup0}. The second estimate in \eqref{bdg.f10} is due to $\E_k[f]_{k,n}=\E_k\wei{f}_{k,n-1}+|df_k|^2$.
		We note that $[f]_{k,n}\le[f]_n$ and 
		and $f^*_{k,n}\le 2f^*_n$. Hence, applying Doob's martingale inequality, we obtain \eqref{bdg.f11}-\eqref{bdg.fweisup1} from \eqref{bdg.f10}-\eqref{bdg.fweisup0}.  The exponential integrability of $V/\Phi$ is a direct consequence of \eqref{est.jnexpwei}.
	\end{proof}
	We observe that necessary conditions for \cref{thm.dismart} are the identities \eqref{id.isometry} and \eqref{id.condiso}, which are valid as long as $\E_i df_j df_k=0$ whenever $i\le j< k$. In a more general case, one would need to control $\E_i\sum_{j<k} df_j df_k$ by a suitable weight, however we have not explored this idea.
	Results for continuous square integrable martingales are stated without proof.
	\begin{theorem}\label{thm.contmart}
		Let $(X_t)_{t\ge0}$ be a square integrable continuous martingale with quadratic variation $(\wei{X}_t)_{t\ge0}$. Define $X^*_{s,\tau}=\sup_{s\le r\le \tau}|X_t-X_s|$.
		There is a universal constant $\lambda>0$  such that for every $\tau>0$,
		\begin{align}
			\sup_{\varepsilon>0}\E \exp\left\{\lambda\sup_{t\le \tau}\frac{|X_t-X_0|}{\varepsilon+\sup_{s\in[0,t]}(\E_s\wei{X}_{s,\tau})^{1/2}}\right\}\le M
			\\\shortintertext{and}
			\sup_{\varepsilon>0}\E \exp\left\{\lambda\sup_{t\le \tau}\frac{\wei{X}_t}{\varepsilon+\sup_{s\in[0,t]}\E_s|X^*_{s,\tau}|^2}\right\}\le M.
		\end{align}		
	\end{theorem}
	Estimates \eqref{bdg.f10}-\eqref{bdg.fweisup0} for the range $m\in(0,1)$ and the exponential integrability appear to be new, as far as the author's knowledge. Estimates with the conditional square functions can be found in \cite{MR0448538,MR365692,MR365692,MR1127721}. 
	Inequality \eqref{bdg.rosenthal} is also known as Burkholder--Rosenthal inequality which is valid for $m\in[1,\infty)$ (\cite{MR365692}). Although we did not recover the case $m=1$ from John--Nirenberg inequality, we have obtained its weak form, namely \eqref{bdg.f10}, which is valid for all $m\in(0,\infty)$.
	The exponential estimates in \cref{thm.dismart,thm.contmart,cor.maximaldagger} are complementary to known estimates for self-normalized processes in \cite[Chapter 12]{MR2488094}.


\section{Weighted VMO processes} 
	\label{sec:weighted_vmo_spaces}
	
	\khoa{apply to stopped BM and continuous martingales.}	
	\begin{definition}
		Let $(V_t)_{t\in[0,\tau]}$ be a process in $\bmo_\phi$, $p\in[1,\infty)$ and $\alpha\in(0,1]$. 
		\begin{enumerate}[(i)]
			\item $V$ is $\vmo_\phi$ if $\kappa^\phi_{0,\tau}(V)=0$.
			\item $V$ is $\vmo_\phi^{p-\var}$ if $V$ is $\vmo_\phi$ and $\rho^\phi(V)$ has finite $p$-variation. 
			\item $V$ is $\vmo_\phi^{\alpha}$ if $V$ is $\vmo_\phi$ and $\rho^\phi(V)$ is $\alpha$-H\"older continuous on the diagonal.
		\end{enumerate}
		We define 
		\begin{align}\label{def.norm.var}
			[V]_{\vmo^{p-\var}_\phi;[0,\tau]}:=\left( \sup_{\pi\in\cpp([0,\tau])}\sum_{[s,t]\in \pi}|\rho^\phi_{s,t}(V)|^{p}\right)^{1/p}
		\end{align}
		and
		\begin{align}\label{def.norm.holder}
			[V]_{\vmo^\alpha_\phi;[0,\tau]}:=\sup_{0\le s<t\le \tau}\frac{\rho^\phi_{s,t}(V)}{(t-s)^\alpha},
		\end{align}
		where $\cpp([0,\tau])$ is the set of all partitions on $[0,\tau]$.
	\end{definition}
	It is evident that $\vmo^\alpha_\phi\subset \vmo^{1/\alpha-\var}_\phi$ and $[V]_{\vmo_\phi^{1/\alpha}-\var;[0,\tau]}\le \tau^\alpha[V]_{\vmo_\phi^\alpha;[0,\tau]}$.  Reasoning as in \cite{le2022stochastic}, processes in $\vmo_\phi$ are necessarily continuous. 
	\begin{proposition}
		Let $(V_t)_{t\in[0,\tau]}$ be a process in $\vmo^{p-\var}_\phi$ and define the function
		\begin{align*}
			(s,t)\mapsto w^\phi_{s,t}(V):=([V]_{\vmo^{p-\var}_\phi;[s,t]})^p.
		\end{align*}
		Then $w(V):\{(s,t)\in[0,\tau]^2:s\le t\}\to[0,\infty)$ is a control.
	\end{proposition}
	\begin{proof}
		For every $s\le u\le t$, we have
		\begin{align*}
			\| \E_s(\phi^*_{s,t})^{-1}|V_t-V_s|\|_\infty \le\|\E_s(\phi^*_{s,u})^{-1} |V_u-V_s|\|_\infty+\|\E_u (\phi^*_{u,t})^{-1}|V_t-V_u|\|_\infty
		\end{align*}
		so that $\rho^\phi_{s,t}(V)\le \rho^\phi_{s,u}(V)+\rho^\phi_{u,t}(V)$.
		The rest of the proof proceeds in the same way as in \cite[Section 3]{le2022stochastic}.
	\end{proof}
	\begin{theorem}
		Let $(V_t)_{t\in[0,\tau]}$ be $\vmo_\phi$. Then
		\begin{align}\label{est.jnvmophi}
			\sup_{r\in[0,\tau]}\left\|\E_r \exp\left(\lambda\sup_{t\in[r,\tau]}\frac{|V_t-V_r|}{\phi^*_{r,t}}\right)\right\|_\infty<\infty \text{ for every }\lambda>0.
		\end{align}
		If $V$ is in $\vmo^{p-\var}_\phi$ for some $p\in(1,\infty)$ then there are constants $c_p,C_p$ such that
		\begin{gather}
			\label{est.phi1}
			\sup_{r\in[0,\tau]}\left\|\E_r\exp\left(\lambda \sup_{t\in[r,\tau]}\frac{|V_t-V_r|}{\phi^*_{r,t}} \right)\right\|_\infty\le M^{1+(e^3\lambda)^pw^\phi_{0,\tau}(V)} \text{ for every }  \lambda>0,
			\\
			\E\exp\left(\lambda \sup_{t\in[0,\tau]}\frac{|V_t-V_0|^{p'}}{|\phi^*_{0,t}|^{p'}} \right)<\infty \text{ whenever }  \lambda(w^\phi_{0,\tau}(V))^{\frac1{p-1}}<c_p,
			\label{est.phi2super}
			\\\left\|\E_s\sup_{t\in[r,\tau]}\frac{|V_t-V_r|^m}{|\phi^*_{r,t}|^m}\right\|_\infty\le C_p \Gamma( m/{p'}+1)(w^\phi_{s,t}(V))^{m/p} \text{ for every }m\ge1.
		\end{gather}
		In \eqref{est.phi2super}, we have set $p'=p/(p-1)$.
		If $V$ is in $\vmo^{1-\var}_\phi$ then \eqref{est.phi1} still holds with $p=1$ and
		\begin{align}
			\PP(|V_t-V_s|\le e^3 \phi^*_{s,t} w_{s,t}^\phi(V)\text{ for all } s\le t\le \tau)=1.
		\end{align}
	\end{theorem}
	\begin{proof}
		\eqref{est.jnvmophi} is a direct consequence of \eqref{est.jnexpwei}. We only show \eqref{est.phi1} because the other estimates are derived as in \cite[Section 3]{le2022stochastic}.
		We define $t_0=0$ and for each integer $k\ge1$, 
		\[
			t_k=\sup\{t\in[t_{k-1},\tau]:\lambda |w^\phi_{t_{k-1},t}(V)|^{1/p}\le e^{-3}\}.
		\]
		We have $\lambda \kappa^\phi_{t_{k-1},t_k}(V)\le \lambda \rho^\phi_{t_{k-1},t_k}(V)\le \lambda |w^\phi_{t_{k-1},t_k}(V)|^{1/p}\le e^{-3}$. The argument in the proof of \cref{thm.bmophi} shows that the left-hand side of \eqref{est.phi1} is at most $M^n$ with $M$ defined in \eqref{def.M}. 
		By continuity of $w^\phi$, we have $\lambda^p w^\phi_{t_{k-1},t_k}=e^{-3p}$ for $k=1,\ldots,n-1$ and $\lambda^p w^\phi_{t_{n-1},t_n}\le e^{-3p}$. By definition of controls, we have
		\begin{align*}
			\frac{n-1}{(e^{3}\lambda)^p}\le\sum_{k=1}^n w^\phi_{t_{k-1},t_k}(V)\le w^\phi_{0,\tau}(V),
		\end{align*}
		which implies that $n\le 1+(e^{3}\lambda)^p w^\phi_{0,\tau}(V)$. 
	\end{proof}
	\begin{corollary}
		Let $(V_t)_{t\in[0,\tau]}$ be a process in $\bmo_\phi^{p-\var}$ with $p\in(1,\infty)$. Assume that there is $q\in(1,\infty)$ such that
		\begin{align*}
			\E \exp\left({\lambda|\phi_\tau|^{\frac1{q-1}}}\right)<\infty \text{ for every }\lambda>0.
		\end{align*}
		Then
		\begin{align*}
			\E\exp\left(\lambda\sup_{t\in[0,\tau]}|V_t-V_0|^{\frac{p'}{q}}\right)<\infty \text{ for every }\lambda>0.
		\end{align*}
	\end{corollary}
	\begin{proof}
		If $w^\phi_{0,\tau}(V)=0$, $V$ is  a constant and the claim is trivial. We thus assume that $w^\phi_{0,\tau}(V)\neq0$.
		Put $X=\sup_{t\in[0,\tau]}|V_t-V_0|$, $Y=\phi^*_{0,\tau}$. By Doob's maximal inequality, we have
		\begin{align*}
			\E e^{\lambda |\phi^*_{0,\tau}|^{1/(q-1)}}\le C \E e^{\lambda \phi_\tau^{1/(q-1)}}<\infty \text{ for every }\lambda>0.
		\end{align*}
		Let $\lambda,\varepsilon$ some positive numbers.
		Using Young inequality
		\begin{align*}
		 	\lambda X^{\frac{p'}q}\le \frac {\varepsilon^q}q\left( \frac{X}{Y}\right)^{p'}+\frac{(\varepsilon^{-1}\lambda)^{q'}}{q'}Y^{\frac {q'}q}
		 \end{align*}
		  and H\"older inequality, we have
		\begin{align*}
			\E e^{\lambda X^{p'/q}}\le \left[\E e^{\varepsilon^q (\frac XY)^{p'}}\right]^{1/q}\left[\E e^{(\varepsilon^{-1}\lambda)^{q'}Y^{q'/q}}\right]^{1/q'}.
		\end{align*}
		Observing that $\frac XY\le \sup_{t\in[0,\tau]}\frac{|V_t-V_0|}{\phi^*_{0,t}}$ and $q'/q=1/(q-1)$, we can choose  $\varepsilon$ sufficiently small and apply \eqref{est.phi2super} to obtain the result.	
	\end{proof}

\bibliographystyle{alpha}
\bibliography{biblio}
\end{document}